\documentclass[a4paper, 11pt]{article}

\usepackage{calc, amsmath, amsthm, a4, latexsym, amssymb, color, url, soul}
\usepackage{graphicx}
\definecolor{comcolor}{rgb}{0.9,0.3,0.3}
\definecolor{starcolor}{rgb}{0.3,0.3,0.9}
\definecolor{hscolor}{rgb}{0.9,0.6,0.5}
\definecolor{darkgreen}{rgb}{0.1,0.6,0.3}

\newtheorem{thm}{Theorem}[section]
\newtheorem{lemma}[thm]{Lemma}

\newtheorem{prop}[thm]{Proposition}

\theoremstyle{definition}
\newtheorem{defn}[thm]{Definition}

\newtheorem{rem}[thm]{Remark}

\newcommand{\be}[1]{\begin{equation}\label{#1}}
\newcommand{\ee}{\end{equation}}
\newcommand{\ba}{\begin{array}}
\newcommand{\ea}{\end{array}}
\newcommand{\bal}{\begin{aligned}}
\newcommand{\eal}{\end{aligned}}

\newcommand{\R}{\mathbb{R}}

\newcommand{\N}{\mathbb{N}}

\newcommand{\E}{\mathbb{E}}

\renewcommand{\P}{\mathbb{P}}

\newcommand{\1}{1\hspace{-0.098cm}\mathrm{l}}

\newcommand{\dd}{{\text{d}}}

\begin{document}

\begin{center}
{\Large \bf Structural properties of the seed bank and the two island diffusion.}\\[5mm]

\vspace{0.7cm}
\textsc{Jochen Blath$^1$, Eugenio Buzzoni$^1$, Adri\'an Gonz\'alez Casanova$^2$, Maite Wilke Berenguer$^1$} 

\vspace{0.5cm}
$^1$ Institut f\"ur Mathematik, Technische Universit\"at Berlin, Germany.\\
$^2$ Weierstra\ss\ Institut Berlin, Germany and Instituto de Matem\'aticas, Universidad Nacional Aut\'onoma de M\'exico,  Mexico.\\
\end{center}

\begin{abstract}
\noindent We investigate various aspects of the (biallelic) {\em Wright-Fisher diffusion with seed bank} in conjunction with and contrast to the \emph{two-island model} analysed e.g.\ in \cite{KZH08} and \cite{NG93}, including moments, stationary distribution and reversibility, for which our main tool is duality. Further, we show that the Wright-Fisher diffusion with seed bank can be reformulated as a one-dimensional stochastic delay differential equation, providing an elegant interpretation of the age structure in the seed bank also forward in time in the spirit of \cite{KKL01}. We also provide a complete boundary classification for this two-dimensional SDE using martingale-based reasoning known as McKean's argument.

  \par\medskip
\end{abstract}


\section{Introduction}
\label{sec:intro}

The \emph{Wright-Fisher diffusion} is a classical probabilistic object in mathematical population genetics. It describes  the scaling limit of the neutral allele frequencies in a large haploid population on a macroscopic timescale, known es \emph{evolutionary timescale}. This diffusion process and its generalizations have been thoroughly investigated, starting with the pioneering work of Wright \cite{W31}. See also \cite{EK92, Fu03, Et11} for more recent mathematical accounts, in particular regarding characterizations of the stationary distribution in the presence of weak neutral mutation and corresponding boundary classification, and for further references. It is well known that the scaling limit of the genealogy of a sample from this model is given by a \emph{Kingman coalescent} (with mutation), cf.\  \cite{W08} for an overview of coalescent theory.

In the presence of population structure, e.g.\ in the guise of the two-island model \cite{W31, M59}, many new effects appear.
In particular, the genealogy of a sample taken from the subdivided population will now be described by a \emph{structured coalescent} instead of the classical Kingman coalescent, in which two lines may merge only at times when both are in the same island.  Yet, other qualitative features remain unchanged, including the fact that the structured coalescent with two islands still ``comes down from infinity'', and that the Wright-Fisher diffusion with two islands (without mutation) will eventually fixate. In this model, there seems to be no explicit characterization of the stationary distribution, though recursion formulas may still be found, see e.g.\ \cite{NG93, Fu03, KZH08} for results in this direction. We are also not aware of a full boundary classification. The standard Feller approach via speed measure and scale function cannot be employed here, since the two island model leads to a two-dimensional diffusion.

The situation changes further when a strong seed bank is being added to a classical Wright-Fisher model. Scenarios with seed bank are less well analyzed, and in fact only recently, in \cite{BGCKW16}, the \emph{Wright-Fisher diffusion with seed bank} and its dual, the {\em seed bank coalescent}, have been introduced as mathematical objects (see also \cite{LM15}, in which the same dual has been obtained as scaling limit of the genealogy in a metapopulation model with peripatric speciation). While at first glance similar to the two-island model and the structured coalescent, the seed bank diffusion and its dual exhibit some remarkable qualitative differences. For example, the seed bank coalescent \emph{does not} come down from infinity, and its expected time to the most recent common ancestor is unbounded as the sample size increases (see \cite{BGCKW16} for details). Hence it is a natural task to investigate the properties and relation between these models.

The paper is organized as follows: The basic models under consideration are introduced in Section \ref{ssn:thebasicmodel} followed by the characterization of the seed bank diffusion as a {\em stochastic delay differential equation} in  Section \ref{ssn:generalizations}, providing an elegant manifestation of the age-structure introduced by seed banks in a forward-in-time model similar to that seen in the backward-in-time considerations in the classical modeling of seed banks in \cite{KKL01}.  In Section \ref{sec:duality}, we observe a non-standard dual processes for our models with mutation that allows us to characterize the moments of the unique stationary distribution with the help of recursions and show that the seed bank diffusion is non-reversible. 
Finally, in Section \ref{sec:boundary}, we investigate the boundary behavior of the seed bank and the two island diffusion using a technique called McKean's argument, which is based on the martingale convergence theorem on stochastic intervals and is suitable also in multi-dimensional settings. A complete analysis of both models is possible, since they are instances of so-called {\em polynomial diffusions}, which have recently drawn considerable attention in the finance literature (see e.g.\ \cite{FL16}). We think that the flexibility of McKean’s argument should potentially make it widely applicable in population genetics, beyond the two models descussed above.

\subsection{The model(s)}
\label{ssn:thebasicmodel}

The \emph{Wright-Fisher diffusion with seed bank} was recently introduced in \cite{BGCKW16} as the forward in time scaling limit of a bi-allelic Wright-Fisher model (with type space $\{a,A\}$) that describes a population where individuals may stay inactive in a \emph{dormant form} such as seeds or spores (in the \emph{seed bank}), essentially ``jumping'' a significant (geometrically distributed) number of generations, before rejoining the \emph{active} population. For an active population of size $N$ and a seed bank size $M=\lfloor K  N \rfloor$,  $K>0$, under the classical scaling of speeding time by a factor $N$ the $a$-allele frequency process $(X^N_t)_{t \geq 0}$ in the active and $(Y^N_t)_{t \geq 0}$ in the dormant population converge to the (unique strong) solution $(X_t,Y_t)_{t\geq 0}$ of a two-dimensional SDE. In \cite{BEGCKW15} the model was extended to include mutation in both the active and the dormant population in which case the limiting process is the solution to the SDE given in Definition \ref{defn:system} below. Since the population model and limiting result are completely analogous to the case without mutation we refrain from details and instead refer to \cite{BGCKW16}, Section 2. 

\begin{defn}[Seed bank diffusion]
\label{defn:system}
Let $(B(t))_{t\geq 0}$ be a standard Brownian motion, $u_1,u_2,u_1',u_2'$ be finite, non-negative constants and $c, K$ finite, positive constants. The {\em Wright-Fisher diffusion with seed bank} with parameters $u_1,u_2,u_1',u_2',c,K$, starting in $(x, y)\in [0,1]^2$, is given by the $[0,1]^2$-valued continuous strong Markov process $(X(t), Y(t))_{t \geq 0}$ that is the unique strong solution of the initial value problem 
\begin{align}
\label{eq:system}
\text{d} X(t) & = \big[-u_1X(t) +u_2(1-X(t))+ c(Y(t) -X(t))\big]\text{d}t + \sqrt{X(t)(1-X(t))}\text{d}B_t, \notag \\[.1cm]
\text{d} Y(t) & = \big[-u_1'Y(t) +u_2'(1-Y(t)) + Kc(X(t) -Y(t))\big]\text{d}t,
\end{align}
with $(X(0), Y(0)) =(x,y) \in [0,1]^2$. 
\end{defn}
 
The first coordinate process $(X(t))_{ t \geq 0}$ can be interpreted as describing the fraction of $a$-alleles in the limiting \emph{active population}, while $(Y(t))_{ t \geq 0}$ gives the fraction of $a$-alleles in the limiting \emph{dormant population}, i.e.\ in the seed bank. 
The parameters $u_1, u_2$ describe the mutation rates from $a$ to $A$,  respectively from $A$ to $a$, in the active population, and $u_1', u_2'$ the corresponding values in the seed bank. Note that the mutation rates may differ for active and dormant individuals. $K$ fixes the so-called \emph{relative seed bank size} ($M=\lfloor KN\rfloor$) and $c$ is the rate of migration between the active population and seed bank, i.e.\ initiation of dormancy and  resuscitation. 
For more details on the biological background see \cite{BGCKW16} and \cite{BEGCKW15}.

\begin{rem}[General model and two-island diffusion]
\label{rem:generalmodel}  
A natural extension of this model can be obtained by (potentially) adding noise in the second coordinate. For parameters $u_1,u_2,u_1',u_2', \alpha, \alpha' \geq 0$, $c,c'>0$ and independent standard Brownian motions $(B(t))_{t\geq 0}, (B'(t))_{t\geq 0}$ consider the initial value problem 
\begin{align}
\label{eq:system_general}
\text{d} X(t) & = \big[ \!\!-u_1X(t) +u_2(1\!-X(t)) + c(Y(t)\! -X(t))\big]\text{d}t \!+ \alpha \sqrt{X(t)(1-X(t))}\text{d}B(t), \notag \\[.1cm]
\text{d} Y(t) & = \big[\!\!-u_1'Y(t) +u_2'(1\!-Y(t)) + c'(X(t)\! -Y(t))\big]\text{d}t \!+ \alpha' \sqrt{Y(t)(1-Y(t))}\text{d}B'(t),
\end{align}
with $(X(0), Y(0)) =(x,y) \in [0,1]^2$. 

For $\alpha = 1$, $\alpha'=0$ and $c'=cK$ (for a $K>0$) this is the seed bank diffusion. 
For $\alpha'>0$ we obtain the diffusion of Wright's \emph{two-island model} initially introduced in \cite{W31} and considered in this form for example in \cite{KZH08}.
\end{rem}
As is standard, Theorem 3.2 in \cite{SS80} provides the existence of a unique strong solution for every (possibly random) initial condition $(X(0), Y(0))=(x,y) \in [0,1]^2$, both initial value problems \eqref{eq:system} and \eqref{eq:system_general} admit a unique strong semimartingale solution which is a two-dimensional continuous strong Markov diffusion. Denote by $A^{u_1,u_2,u_1',u_2', \alpha, \alpha',c,c'}$ its Markov generator and note that its domain contains $C^2(\lbrack 0,1 \rbrack^2)$, the space of twice continuously differentiable functions inside $\lbrack 0,1 \rbrack^2$ (continuous at the boundary). Since the model referred to will be clear from the context, we will usually omit the superscripts on the generator and simply write $A$. The action of $A$ on any test function $f \in C^2(\lbrack 0,1 \rbrack^2)$ is then described by
\begin{align}
\label{eq:generator_general}
Af(x,y)=&\big[-u_1x +u_2(1-x)+ c(y - x)\big]\frac{\partial f}{\partial x}(x,y) + \frac{\alpha^2}{2} x(1-x)\frac{\partial^2 f}{\partial x^2}(x,y) \notag \\ 
& + \big[-u_1'y +u_2'(1-y) + c'(x - y)\big]\frac{\partial f}{\partial y}(x,y)+ \frac{(\alpha')^2}{2} y(1-y)\frac{\partial^2 f}{\partial y^2}(x,y)
\end{align}
Note that there is no ambiguity in the definition of the process at the boundaries, since the diffusion part vanishes at 0 and 1, so that no further conditions on the domain of the generator are required. 

\begin{rem}[Extension to multiple seed banks]
\label{rem:multiple_seed_banks}
It is straightforward to extend system \eqref{eq:system} to several (e.g.\ geographically) subdivided seed banks. Indeed, let $k \ge 1$ and denote the frequency process for the active population by $(X(t))_{t\geq 0}$. Assume there are $k$ seed banks. For each seed bank $i \in \{1, \dots, k\}$ we consider specific parameters $c_i, K_i$ as well as mutation rates $u_1^i, u_2^i$ and denote its frequency process by $(Y_i(t))_{t\geq 0}$. Then the {\em seed bank diffusion with $k$ seed banks} is given by the following $k+1$ interacting SDEs  
\begin{align}
\label{eq:system3}
\text{d} X(t) &  \! = \big[-u_1X(t) +u_2(1 \! - \! X(t))+ \! \sum_{i=1}^k c_i (Y_i(t) \! - \! X(t))\big]\text{d}t+ \! \sqrt{X(t)(1-X(t))}\text{d}B(t), \notag \\
\text{d} Y_i(t) & \! = \big[\!\!-u_1^i Y_i(t) +u_2^i (1 \! - \! Y_i(t)) + K_i c_i(X(t) \!  - \! Y_i(t))\big]\text{d}t, \quad i \in \{1, \dots, k\},
\end{align}
with initial value $(X(0), Y_1(0), \dots, Y_k(0))=(x, y_1, \dots, y_k) \in [0,1]^{k+1}$. Existence and uniqueness are again standard with Theorem 3.2 in \cite{SS80}. Of course, this can be further generalized to multiple islands with multiple seed banks, see \cite{DHP16}.
\end{rem}

\subsection{A stochastic delay differential equation} 
\label{ssn:generalizations}

Note that the only source of randomness in the two-dimensional system $\eqref{eq:system}$ and also in its generalization $\eqref{eq:system3}$ is the one-dimensional Brownian motion $(B(t))_{t\geq 0}$ driving the fluctuations in the active population. This and the special form of the seed bank(s) allow us to reformulate this system as an essentially one-dimensional {stochastic delay differential equation}. Recall the notation from Remark \ref{rem:multiple_seed_banks} and abbreviate, for convenience, $u^i:=u_1^i+u_2^i$, $i \in \{1, \ldots, k\}$.

\begin{prop} 
\label{prop:sdde}
The solution to \eqref{eq:system3} with initial values $x,y_1, \dots, y_k\in [0,1]$ is a.s.\ equal 
to the solution of the unique strong solution of system of \emph{stochastic delay differential equations} 
\begin{align}
\label{eq:sdde}
\text{d}X(t)&=   \sum_{i=1}^k  c_i \bigg( y_i e^{-(u^i +K_i c_i)t}+\int_0^t e^{-(u^i+K_i c_i)(t-s)}(u_2^i+K_i c_i X(s)) \text{d}s - X(t) \bigg) \text{d}t \notag \\ 
& \quad \qquad + \big[-u_1X(t) +u_2(1-X(t))\big] \text{d}t + \sqrt{X(t)(1-X(t))}\text{d}B(t), \notag \\
dY_i(t)&= \bigg( -y_i(u^i +K_i c_i)e^{-(u^i +K_i c_i)t} \notag\\
  & \qquad  \;\; -(u^i+K_i c_i)\int_0^te^{-(u^i+K_i c_i)(t-s)}(u_2^i + K_i c_i X(s))\dd s +u_2^i + K_i c_i X(t) \bigg) \text{d} t,
\end{align}
for $i \in \{1, \dots, k\}$ with the same initial condition.
\end{prop}

Note that the the first equation in \eqref{eq:sdde} does not depend on $Y_i, i=1, \dots, k$, and that the latter equations for the $Y_i$ are in turn deterministic functions of $X$, so that the system of SDDEs is essentially one-dimensional.

\begin{rem}
Let us consider a simple special case of the above result to reveal the underlying ``age structure'': It is an immediate corollary from the above that the seed bank diffusion solving $\eqref{eq:system}$ with parameters $c=1$,  $K=1$, $u_1=u_2=u_1'=u_2'=0$, started in $X(0)=x=y=Y(0) \in [0,1]$ is a.s.\ equal to the unique strong solution of the stochastic delay differential equations
\begin{align}
\label{eq:sdde_simplified}
\text{d}X(t) &= \bigg( x e^{-t} + \int_0^t e^{-(t-s)} X(s)  \text{d}s  - X(t) \bigg) \text{d} t + \sqrt{X(t)(1-X(t))}\text{d}B(t), \notag \\
\text{d}Y(t) &= \bigg( -y e^{-t} - \int_{0}^{t} e^{-(t-s)} X(s) \text{d}s +X(t)\bigg) \text{d} t,
\end{align}
with the same initial condition. This now provides an elegant interpretation of the \emph{delay} in the SDDE as the seed bank. Indeed, it shows that the \emph{type} ($a$ or $A$) of any ``infinitesimal'' resuscitated individual, is determined by the \emph{active} population present an \emph{exponentially distributed time ago} (with a cutoff at time 0), \emph{which the individual spent dormant in the seed bank}. The net effect is positive if the frequency of $a$-alleles at that time was higher than the current frequency, and negative if it was lower. This is the forward-in-time equivalent of the model for seed banks or dormancy in the the coalescent context as formulated in \cite{KKL01}, where the seed bank is modelled by having individuals first choose a generation in the past according to some measure $\mu$ and then choosing their ancestor uniformly among the individuals present in that generation. The seed bank model given in \cite{BGCKW16} is obtained when $\mu$ is chosen to be geometric, i.e.\ memoryless, like the exponential distribution. This indicates that a forward-in-time model for more general dormancy models are to be searched among SDDEs rather than among SDEs. 

Such a reformulation is of course not feasible for the two island model, which is driven by two independent sources of noise. 
\end{rem}

\begin{proof}[Proof of Proposition \ref{prop:sdde}]
Recall e.g.\ from \cite[Proposition 3.1]{RY99} that for continuous semimartingales
$Z,W$ we have the integration by parts formula
$$
\int_0^t W(s) dZ(s) = W(t)Z(t) - Z(0) W(0) - \int_0^t Z(s) dW(s) - \langle Z,W \rangle(t),
$$
where $\langle \cdot, \cdot \rangle$ denotes the covariance process and $t\ge 0$. 
Note that for every differentiable deterministic function $f$, since $\langle Z, f \rangle \equiv 0,$ this reduces to
$$
f(t)Z(t)-f(0)Z(0)=\int_0^t f(s)\text{d}Z(s) +\int_0^t f'(s) Z(s) \text{d}s.
$$
Substituting the expression for $\text{d}Y_i(t)$ from \eqref{eq:system3}, we obtain that
\begin{align}
\label{eq:f(t)X_2(t)}
f(t)Y_i(t)-f(0)Y_i(0) \notag
=&\int_0^t f(s)\big[-u_1^i Y_i(s) +u_2^i (1-Y_i(s)) + K_i c_i(X(s) -Y_i(s))\big]\text{d}s \\ &+\int_0^t f'(s) Y_i(s) \text{d}s.
\end{align}
Letting
$
f(t):=e^{(u^i+K_i c_i)t}, t \ge 0,
$
equation \eqref{eq:f(t)X_2(t)} further simplifies to 
\begin{equation*}
\label{eq:f(t)X_2(t)2}
f(t)Y_i(t) - f(0)Y_i(0) =\int_0^t  e^{(u^i+K_i c_i)s}(K_i c_iX(s)+u_2^i) \text{d}s.
\end{equation*}
This can be rewritten, given the initial value $y_i=Y_i(0)$, as
$$ 
e^{(u^i+K_i c_i)t}Y_i(t)=y_i+\int_0^t  e^{(u^i+K_i c_i)s}(K_i c_iX(s)+u_2^i) \text{d}s.
$$
By dividing on both sides by $ e^{(u^i+K_i c_i)t}$ we finally get
\begin{align}\label{eq:solY}
Y_i(t)=y_i e^{-(u^i+K_i c_i)t}+\int_0^t \ e^{(u^i+K_i c_i)(s-t)}(K_i c_iX(s)+u_2^i) \text{d}s.
\end{align}
Plugging this into the first line of the system in \eqref{eq:system3} proves that the unique strong solution of \eqref{eq:system3} is a strong solution to \eqref{eq:sdde}. On the converse, let $(X(t),Y_1(t),\ldots,Y_k(t))_{t\geq 0}$ now be a solution to \eqref{eq:sdde}. (We already know that there exists at least one.) Using \eqref{eq:solY} we immediately see that $(X(t))_{t\geq0}$ solves the first equation in \eqref{eq:system3}. Likewise, using \eqref{eq:solY} in the right-hand-side of the last $k$ equations in \eqref{eq:sdde}, we obtain the last $k$ equations of \eqref{eq:system3}. Since \eqref{eq:system3} has a unique solution, this must then hold for \eqref{eq:sdde}, too, and the two solutions coincide $\P$-almost surely.
\end{proof}

\section{Duality, moments and stationary distribution}
\label{sec:duality}

A convenient way to study the behavior of diffusions in population genetics has proven to be the usage of (moment-) duality for Markov processes, cf. \cite{JK14} for an overview of the technique and \cite{BGCKW16}, \cite{Et11} for examples of applications. The art lies in finding a suitable dual that can serve as such a tool. In \cite{BGCKW16} the \emph{moment dual of the seed bank diffusion (without mutation)} $(N(t),M(t))_{t\geq 0}$ (known therein as \emph{block-counting process of the seed bank coalescent}) is given as the continuous time Markov chain with values in $E := {\N}_0\times {\N}_0$ equipped with the discrete topology, with conservative generator $\bar A$ given by: 
\begin{align*} 
\bar A_{(n,m), (\bar n,\bar m)} = \begin{cases}
		    \binom{n}{2}& \text{if } (\bar n,\bar m) = (n-1,m),\\
		    cn & \text{if } (\bar n,\bar m) = (n-1,m+1),\\
		    cKm  & \text{if } (\bar n,\bar m) = (n+1,m-1),
               \end{cases}
 \end{align*}            
 for every $(n,m)\in \N_0\times \N_0$ and zero otherwise off the diagonal.   
 
 A difficulty arises when adding \emph{mutation} to the model. There are several ways of incorporating this mechanism into a dual and we comment on this as well as the motivation behind our strategy -- adding a \emph{death state} to the state-space -- below, but will first formally introduce our dual.

\begin{defn}[Moment dual of the general diffusion] \label{defn:Dual_general}
Consider the space $E := {\N}_0\times {\N}_0 \cup \{(\partial,\partial)\}$ equipped with the discrete topology. Let $u_1,u_2,u_1',u_2',\alpha,\alpha' \geq 0$, $c,c'>0$. Define  $(N(t),M(t))_{t\geq 0}$ to be the continuous time Markov chain with values in $E$ with conservative generator $\bar A$ given by: 
\begin{align*} 
\bar A_{(n,m), (\bar n,\bar m)} = \begin{cases}
		    \alpha^2\binom{n}{2} + nu_2& \text{if } (\bar n,\bar m) = (n-1,m,),\\
		    (\alpha')^2\binom{m}{2}+m u^{'}_2 & \text{if } (\bar n,\bar m) = (n,m-1),\\
		    nu_1+m u^{'}_1 & \text{if } (\bar n,\bar m) = (\partial,\partial) ,\\
		    cn & \text{if } (\bar n,\bar m) = (n-1,m+1),\\
		    c'm  & \text{if } (\bar n,\bar m) = (n+1,m-1),
               \end{cases}
\end{align*}
for every $(n,m)\in \N_0\times \N_0$ and zero otherwise off the diagonal. We call this process the \emph{moment dual of the diffusion associated with the diffusion given in \eqref{eq:system_general}}.
\end{defn}

The name of the process will be justified in Lemma \ref{lemma:duality} below.
This dual arises in the context of a \emph{sampling duality}. See \cite{GCS17}  for a thorough introduction to the concept. It is based on the idea that the question of ``what is the probability of sampling $n$ individuals of type $a$ at time $t$, if the frequency of type $a$ is $x$ at time 0?'' can be answered in two ways: One, looking forward in time at the diffusion which will give the frequency of type $a$ individuals at time $t$ precisely, but also two, tracing back the genealogy to the number of ancestors of the sample present at time 0 and using the frequency $x$. It is precisely in this latter view that one realizes the need of an artificial \emph{death-state}. In order for all $n$ individuals in the sample to be of type $a$ at time $t$ it is imperative that we do not encounter a mutation from type $a$ to $A$ in the forward sense, i.e. a mutation from $A$ to $a$ in the coalescent-time, on their ancestral lines. Hence, the process $(N(t),M(t))_{t\geq0}$ is killed off as soon as this happens, since the probability for the sample to be of type $a$ only is now 0. At the same time, if we encounter a mutation of type $A$ to $a$ in the forward sense, i.e.\ a mutation from type $a$ to $A$ tracing backwards, we are assured all descendants of that line are of type $a$ with probability 1 and we can stop tracing it, whence the process is reduced by one line. See Figure \ref{fig:mutcoalescent} for an illustration.

\begin{figure}\label{fig:mutcoalescent}
\center
\includegraphics[width=.9\textwidth]{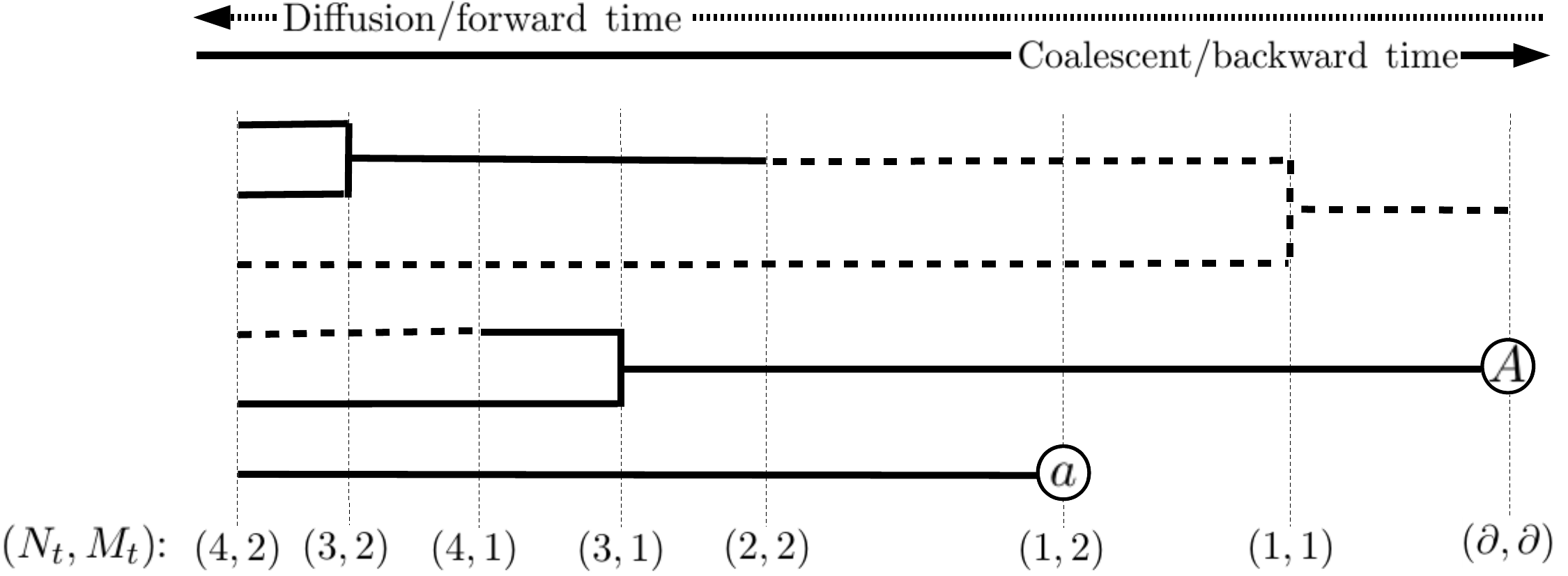}
\caption{The coalescent corresponding to the process defined in Definition \ref{defn:Dual_general}. The dashed and black lines represent the two islands. When a forward-mutation of type $A \mapsto a$ (``$a$'') occurs, the line is ended, since it ensures all its leaves to be of type $a$. A forward-mutation of type $a \mapsto A$ (``$A$'') renders it impossible to have all leaves of type $a$ and the process jumps to the death state $(\partial, \partial)$. }
\end{figure}

It is trivial to extend the dual process in Definition \ref{defn:Dual_general} to a general structured coalescent. A structured-mutation moment dual is new in the literature, as far as we know. These moment duals differ from the \emph{weighted} moment dual for the Wright-Fisher diffusion with mutation introduced in \cite{EG9} and studied extensively in \cite{GJL17} and \cite{EGT10}. The small difference between our construction for mutation and the construction in \cite{EG9}, namely the addition of the extra state $\partial$, makes our dual compatible with the presence of selection as in \cite{KN97}.

The following are straightforward, but important observations on the duals: Note that in the case of $u_1+u_2+u_1'+u_2'>0$, the moment dual of the general diffusion will reach either $\{(0,0)\}$ or $\{(\partial,\partial)\}$ in finite time (for any starting point $(n,m) \in E$), whereas for $u_1+u_2+u_1'+u_2'=0$ it will reach the set $\{(1,0),(0,1)\}$ in finite time ($\P$-a.s.)\ and then alternate between these two states. Furthermore observe that, whenever the dual of the general diffusion is started in some $(n,m)\in E$, it will stay in $\{0, \ldots, n+m\}\times\{0, \ldots, n+m\}\cup\{(\partial,\partial)\}$, hence the state space in this case is, indeed, finite.

\begin{lemma}
\label{lemma:duality}
Let $S:[0,1]^2\times E\rightarrow [0,1]$ be defined as 
\begin{align*}
S((x,y),(n,m)):=x^{n}y^{m}\1_{\N_0\times\N_0}((n,m)) 
\end{align*}
 for any $(x,y)\in[0,1]^2$ and $(n,m)\in E$ and let $u_1,u_2,u_1',u_2',\alpha,\alpha'\geq 0$, $c,c'>0$. 
Then for every $(x,y) \in [0,1]^2$, $(n,m) \in E$ and for any $t \geq 0$
\begin{align*}
\E^{x,y}\big[S\big((X(t),Y(t)),(n,m)\big)\big]=\E_{n,m}\big[S\big((x,y),(N(t),M(t))\big)\big],
\end{align*}
where $(N(t),M(t))_{t \geq 0}$ is defined in Definition \ref{defn:Dual_general} and $(X(t),Y(t))_{t \geq 0}$ is the solution to the SDE in equation \eqref{eq:system_general}.
\end{lemma}

\begin{proof}
Since $S:[0,1]^2\times E \rightarrow [0,1]$ is continuous (in the product topology of $[0,1]^2\times E$), the result follows by proving the assumptions of Proposition 1.2 in \cite{JK14}, which consist of  certain requirements on the respective generators $A$ of $(X(t),Y(t))_{t \geq 0}$ and $\bar A$ of $(N(t),M(t))_{t \geq 0}$.

Recall the generator $A$ of $(X(t),Y(t))_{t\geq 0}$ from \eqref{eq:generator_general} and observe that for any bounded function $h:E \rightarrow \R$, the generator of $(N(t),M(t))_{t\geq 0}$ is given by $\bar A h((\partial,\partial)) = 0$ and 
\begin{align*}
\bar A h (n,m) 	& = \left[\alpha^2\binom{n}{2}+nu_2\right] [h(n-1,m)-h(n,m)] \1_{\N}(n)\\
			& \quad +  \left[(\alpha')^2\binom{m}{2}+mu_2'\right][h(n,m-1)-h(n,m)]\1_{\N}(m)\\
			& \quad + c[h(n-1,m+1)-h(n,m)]\1_{\N}(n)+ c'[h(n+1,m-1)-h(n,m)]\1_{\N}(m) \\
			& \quad + [nu_1+mu_1'][h(\partial,\partial)-h(n,m)],
\end{align*}
for any $(n,m)\in \N_0\times \N_0$ with the convention that $\binom{1}{2}=0$. Let $P$ and $\bar P$ be the semigroups corresponding to $A$ and $\bar A$ respectively. Since $(N(t)+M(t))_{t \geq 0}$ is monotonically non-increasing, the assumptions that $S\big((x,y),(n,m)\big), P_tS \big((x,y),(n,m)\big)$ are in the domain of $\bar A$ and $S\big((x,y),(n,m)\big), \bar P_tS \big((x,y),(n,m)\big)$ are in the domain of $A$ are readily verified.

As $S((x,y),(\partial,\partial))=0$ for all $(x,y)\in[0,1]^2$, we immediately see that 
\begin{align*}
 AS \big((x,y),(\partial,\partial)\big) = 0 = \bar AS((x,y),(\partial,\partial))
\end{align*}
for any $(x,y)\in[0,1]^2$. Furthermore, if we fix $(x,y) \in [0,1]^2$ and  $(n,m) \in \N_0\times\N_0$,
\begin{align*}
AS \big((x,y),(n,m)\big) & = \left[-u_1x +u_2(1-x)+ c(y -x)\right] nx^{n-1}y^m \\
			  & \qquad + \frac{\alpha^2}{2}x(1-x)n(n-1)x^{n-2}y^m\\
			  & \qquad + \left[-u_1' y +u_2'(1-y) + c'(x -y)\right]mx^{n}y^{m-1}\\
			  & \qquad + \frac{(\alpha')^2}{2}y(1-y)m(m-1)x^{n}y^{m-2}\\
			  & = \Big[\alpha^2\binom{n}{2}+nu_2\Big] [x^{n-1}y^{m}-x^{n}y^{m}]\\
			  & \qquad +[(\alpha')^2\binom{m}{2}+mu_2'][x^{n}y^{m-1}-x^{n}y^{m}]\\
			  & \qquad +c[x^{n-1}y^{m+1}-x^{n}y^{m}]+ c'[x^{n+1}y^{m-1}-x^{n}y^{m}]\\
			  & \qquad + (nu_1+mu_1')[0-x^{n}y^{m}]\\
			  & = \bar AS((x,y),(n,m)).
\end{align*}
\end{proof}

This duality now allows us to use the process $(N(t),M(t))_{t \geq 0}$ to study the mixed moments of $(X(t),Y(t))_{t \geq 0}$ from which we can draw conclusions on the limiting behavior of the diffusions itself. The case with and without mutation differs strongly in this behavior. 

\begin{rem}
 In the absence of mutation in the general diffusion given in \eqref{eq:system_general} with $\alpha=\alpha'=1$
 \begin{align*}
  \lim_{t \to \infty}\mathbb{E}^{x,y} [X(t)^nY(t)^m] = \frac{yc + xc'}{c+c'}
 \end{align*}
for all $(n,m) \in \N_0\times\N_0\setminus\{(0,0)\}$ and all $(x,y) \in [0,1]^2$. From this we can conclude that $(X(t),Y(t))_{t \geq 0}$ converges $\P$-a.s.\ to a random variable $(X_{\infty},Y_{\infty})$ with values in $[0,1]^2$ whose distribution is given by
 \begin{align*}
  \frac{yc + xc'}{c+c'} \delta_{(1,1)} + \frac{(1-y)c + (1-x)c'}{c+c'}\delta_{(0,0)}
 \end{align*}
as can easily be seen by the same arguments as in Proposition 2.9 and Corollary 2.10 in \cite{BGCKW16}.
\end{rem}

\begin{prop}\label{prop:momentprobability}
Let $u_1,u_2,u_1',u_2',\alpha,\alpha'\geq 0$, $c,c'>0$ and assume that at least one mutation rate among $u_1, u_2, u_1', u_2'$ is non-zero. Then, for every $(n,m)\in \N_0\times\N_0$ and for every $(x,y) \in [0,1]^2$
\begin{align*}
\lim_{t \to \infty}\mathbb{E}^{x,y} [X(t)^nY(t)^m]=\P_{n,m}\left\{\lim_{t\rightarrow \infty}(N(t),M(t))=(0,0)\right\}.
\end{align*}
\end{prop}

\begin{proof}
Fix $(x,y) \in [0,1]^2$ and $(n,m) \in \N_0\times\N_0$. Then
\begin{align*}
 \lim_{t \to \infty}\mathbb{E}^{x,y} [X(t)^nY(t)^m]	& = \lim_{t \to \infty}\mathbb{E}^{x,y} \big[\underbrace{X(t)^nY(t)^m\1_{\N_0\times \N_0}(n,m)}_{= S((X(t),Y(t)),(n,m))}\big] \\
							& = \lim_{t \to \infty}\mathbb{E}_{n,m} \left[x^{N(t)}y^{M(t)}\1_{\N_0\times \N_0}(N(t),M(t))\right]\\
							& = \mathbb{E}_{n,m} \left[\lim_{t \to \infty}x^{N(t)}y^{M(t)}\1_{\N_0\times \N_0}(N(t),M(t))\right]\\
							& = \P_{n,m}\left\{\lim_{t\rightarrow \infty}(N(t),M(t))=(0,0)\right\},
\end{align*}
where the last three equalities follow from the duality in Lemma \ref{lemma:duality}, bounded convergence and the fact that $(N(t),M(t))_{t \geq 0}$ is absorbed in $(0,0)$ or $(\partial,\partial)$ \emph{in finite time} $\P$-a.s., respectively. (We use the convention that $0^0=1$.)
\end{proof}
We can now use this to characterize the long-term behavior of the diffusion $(X(t),Y(t))_{t \geq 0}$ solving \eqref{eq:system_general} with mutation. In order to do this, note that Proposition \ref{prop:momentprobability} implies that the following is well-defined.

\begin{defn}\label{def:mixedmoments}
 Let $u_1,u_2,u_1',u_2',\alpha,\alpha'\geq 0$, $c,c'>0$ and assume  $ u_1 + u_2 + u_1' + u_2'>0$ in \eqref{eq:system_general}. For any $(n,m) \in \N_0\times \N_0$ define
 \begin{align*}
  M_{n,m} := \lim_{t \to \infty}\E^{x,y} \lbrack X^n (t) Y^{m} (t) \rbrack\qquad \text{ (for any } (x,y) \in [0,1]^2).
 \end{align*}
\end{defn}

\begin{prop}\label{prop:invariantdistribution}
Let $u_1,u_2,u_1',u_2',\alpha,\alpha'\geq 0$, $c,c'>0$. Assume $ u_1 + u_2 + u_1' + u_2'>0$ in \eqref{eq:system_general}. Then there exists a unique invariant distribution $\mu$ for $(X(t),Y(t))_{t\geq 0}$ and the diffusion is ergodic in the sense that
\begin{align*}
 \P^{x,y}\left\{(X(t),Y(t)) \in \cdot\,\right\} \xrightarrow{\;\;w\;\;} \mu,\quad \text{for $t \to \infty$},
\end{align*}
for all starting points $(x,y) \in [0,1]^2$, where $\xrightarrow{\;\;w\;\;}$ denotes weak convergence of measures. Furthermore $\mu$ is characterized by
\begin{align}\label{eq:invmeasure_moments}
 \forall\, n,m \in \N_0:\qquad \int_{[0,1]^2} x^ny^m \dd \mu(x,y)\; = \;M_{n,m}.
\end{align}
\end{prop}

\begin{proof}
The unique solvability of the moment problem on $[0,1]^2$ yields existence of a unique distribution $\mu$ such that \eqref{eq:invmeasure_moments} holds, which in particular implies 
\begin{align*}
1=M_{0,0}=\int_{[0,1]^2}x^0y^0\dd \mu(x,y)=\mu([0,1]^2). 
\end{align*}
From the definition of the $M_{n,m}, n,m\in\N_0$, we know that
\begin{align*}
 \lim_{t\rightarrow \infty} \int_{[0,1]^2}& p(x,y)\dd\P^{\bar x,\bar y}\left\{(X(t),Y(t)) \in \,\cdot\,\right\} 	\\ 
 & = \lim_{t \to \infty}\E^{\bar x,\bar y} \lbrack p(X(t),Y(t)) \rbrack 
					       = \int_{[0,1]^2} p(x,y) \dd\mu(x,y) 
\end{align*}
for any polynomial $p$ on $[0,1]^2$ (and any $(\bar x,\bar y) \in [0,1]^2$).  Since the polynomials are dense in the set of continuous (and bounded) functions on $[0,1]^2$ we can conclude that 
\begin{align*}
\P^{x,y}\left\{(X(t),Y(t)) \in \cdot\,\right\} \xrightarrow{\;\;w\;\;} \mu. 
\end{align*}
It is now easy to check that $\mu$ is the unique invariant distribution of $(X(t),Y(t))_{t\geq 0}$. 
\end{proof}

Unfortunately, we cannot calculate $\mu$ explicitly, but we can give the following characterization of its mixed moments:
\begin{lemma}\label{lemma:moments} 
Let $u_1,u_2,u_1',u_2',\alpha,\alpha'\geq 0$, $c,c'>0$ and assume $ u_1 + u_2 + u_1' + u_2'>0$.
Then $M_{0,0} = 1$ and the following recursion holds for all $(n,m) \in \N_0\times\N_0\setminus\{(0,0)\}$
\begin{align*}
 M_{n,m}=\frac{1}{D_{n,m}}\left(a_nM_{n-1,m} + a_m'M_{n,m-1} + cnM_{n-1,m+1} + c'mM_{n+1,m-1}\right), 
\end{align*}
where
\begin{align*}
 a_n &:= \alpha^2\binom{n}{2} + n u_2, \qquad \qquad  \qquad a_m':= (\alpha')^2\binom{m}{2} + m u_2, \\
 D_{n,m} &:=\alpha^2\binom{n}{2} + (\alpha')^2\binom{m}{2} + (u_2+u_1)n + (u'_1+u'_2)m + cn + c'm.
\end{align*}
We use the notational convention that $\binom{1}{2} = 0$ and $M_{-1,k}=M_{k,-1}=0$ for any $k \in \N$.
\end{lemma}

\begin{proof}
For the process $(N(t),M(t))_{t\geq 0}$ let 
\begin{align*}
\tau:=\inf\{t\geq 0:(N(t),M(t))\neq(n,m)\}. 
\end{align*}
 For any $(n,m)\neq (0,0)$, $\tau$ is a $\P_{n,m}$-a.s.\ finite stopping time. Using Proposition \ref{prop:momentprobability} in the first and third, and the strong Markov property in the second equality we see
\begin{align*}
M_{n,m} & = \P_{n,m}\left\{\lim_{t\rightarrow \infty}(N(t),M(t))=(0,0)\right\}\\
	& = \sum_{(i,j)\in \N_0\times \N_0}\P_{i,j}\left\{\lim_{t\rightarrow \infty}(N(t),M(t))=(0,0)\right\}\P_{n,m}\left\{(N(\tau),M(\tau))=(i,j)\right\}\\
	& = \sum_{(i,j)\in \N_0\times \N_0}M_{i,j}\P_{n,m}\left\{(N(\tau),M(\tau))=(i,j)\right\}.
\end{align*}
Writing out the values of $\P_{n,m}\left\{(N(\tau),M(\tau))=(i,j)\right\}$ finishes the proof.
\end{proof}

\begin{rem}
 Given the existence of an invariant distribution, the question of reversibility arises naturally. The classical Wright-Fisher frequency process with mutation is reversible. However, the diffusion process of the two-island model is not, as shown in \cite{KZH08}. It turns out, that the seed bank diffusion with mutation is not reversible in general, either. Assume for example that $c, u_1, u_2 \neq 0$ and $u_1'=u_1$, $u_2'=u_2$ and recall that for the diffusion to be reversible we would need 
 \begin{align*}
\E^{\mu} \lbrack f(X(t),Y(t))Ag(X(t),Y(t))\rbrack = \E^\mu \lbrack g(X(t),Y(t))Af(X(t),Y(t))\rbrack,  
 \end{align*}
 for all $f,g \in \mathcal C^2([0,1]^2)$. This, however, fails for $f(x,y)=x$ and $g(x,y)=y$ as can be checked calculating recursively the values for the mixed moments. 
\end{rem}

\section{Boundary classification}
\label{sec:boundary}

We begin with the simple observation that in the presence of mutation the marginals of the stationary distribution $\mu$ of the general diffusion \eqref{eq:system_general} have no atoms at the boundaries. To be precise, if we let $(X,Y)=(X(t),Y(t))_{t \geq 0}$ be the solution to \eqref{eq:system_general}, assume $u_1  u_2  u_1'  u_2'>0$
and recall that $\mu$ denotes the unique invariant distribution of $(X(t),Y(t))_{t \geq 0}$, then, for any $t>0$, we have
\begin{align}\label{eq:noatoms}
\mathbb{P}^\mu \left\{ X(t) \in \lbrace 0,1 \rbrace \right\} = \mathbb{P}^\mu \left\{ Y(t) \in \lbrace 0,1 \rbrace \right\} =0.
\end{align}
This is a straightforward extension of the corresponding observation for the two-island model in \cite[Proposition 1]{KZH08} with an entirely analogous proof, which is therefore ommitted here.

In the above, each of the parameters $u_1, u_2, u_1', u_2'$ is responsible for the value of exactly one of the probabilities in \eqref{eq:noatoms}. This correspondence will be further clarified in the following more detailed description of the boundary behavior of the solution to \eqref{eq:system_general}. Define the first hitting time of $X$ of the boundary 0 by
\begin{align*}
 \tau^X_0:=\inf\{t \geq 0 \mid X(t)=0\},
\end{align*}
and define $\tau^X_1$, $\tau^Y_0$ and $\tau^Y_1$ in the same manner. We say that \emph{started from the interior $X$ will never hit 0}, if for every initial distribution $\mu_0$ such that $\mu_0((0,1)^2)=1$ 
\begin{align*}
  \P^{\mu_0}\left( \tau^X_0 < \infty\right) = 0.
\end{align*}
Using analogous formulations for the other cases, we state the main result of this section:
\begin{thm}\label{thm:boundary}
 Let $(X(t),Y(t))_{t \geq 0}$ be the solution to \eqref{eq:system_general} with parameters satisfying $u_1,u_2,u_1',u_2', \alpha,\alpha'\geq 0$ and $c,c'>0$. 
 \begin{enumerate}
  \item Started from the interior $X$ will never hit 0 if and only if $2u_2 \geq \alpha^2$.
  \item Started from the interior $X$ will never hit 1 if and only if $2u_1 \geq \alpha^2$.
  \item Started from the interior $Y$ will never hit 0 if and only if $2u_2' \geq (\alpha')^2$.
  \item Started from the interior $Y$ will never hit 1 if and only if $2u_1' \geq (\alpha')^2$.
 \end{enumerate}
\end{thm}

\begin{rem}
Note that the theorem simply states that for, say, $2u_2 < \alpha^2$, there must \emph{exist} some initial distribution $\bar \mu_0$ such that $\P^{\bar \mu_0}\{\tau^X_0 < \infty\}>0$. However, we will in fact prove the following, more informative statement (and their respective analogs for $(ii)-(iv)$): 
 
 Let $2u_2 <\alpha^2$. Then, for any $s>0$ there exists an $\varepsilon>0$ such that
 \begin{align*}
  \P\left\{\Vert (X(0),Y(0)) - (0,0) \Vert < \varepsilon \right\} =1 \quad \Rightarrow \quad \P\left\{\tau^X_0 \leq s\right\} >0.
 \end{align*}
\end{rem}

To obtain this result, we regard our SDE as a \emph{polynomial diffusion}. These are solutions to (multidimensional) SDEs whose generator maps the set of polynomials of degree $n$ into itself (for any $n \in \N_0$), see for example \cite{FL16, LP16}. As explained in the remark after Definition 2.1 in \cite{FL16}, these are SDES of the form 
\begin{align}\label{eq:poldiff}
 \dd Z(t) = b(Z(t))\dd t + \sigma(Z(t))\dd W(t)
\end{align}
where $W$ is a (multidimensional) Brownian motion, $b$ consists of polynomials of degree at most 1 and $a(x,y) := \sigma(x,y)\sigma(x,y)^T$ of polynomials of degree at most 2.

A quick glance immediately allows the observation that our generalized SDE \eqref{eq:system_general} can be rewritten in the form of \eqref{eq:poldiff} with 
\begin{align*}
 b(x,y) := \begin{pmatrix}
           -u_1x+u_2(1-x)+c(y-x)\\
           -u_1'y+u_2'(1-y)+c'(x-y)
          \end{pmatrix} \text{ and }
 \sigma(x,y) := \begin{pmatrix}
                \alpha \sqrt{x(1-x)} & 0 \\
                0 & \alpha' \sqrt{y(1-y)}
               \end{pmatrix}
\end{align*}
on $[0,1]^2$ and a two-dimensional Brownian motion $W$. Since $b$ consists of polynomials of degree 1 and 
\begin{align*}
 a(x,y) := \sigma(x,y)\sigma(x,y)^T = \begin{pmatrix}
					  \alpha^2 x(1-x) & 0 \\
					  0 & (\alpha')^2 y(1-y)
				      \end{pmatrix}
\end{align*}
of polynomials of degree 2, our SDE is indeed a \emph{polynomial diffusion on $[0,1]^2$} in the sense (and notation) of \cite{FL16} and we are free to use the results found therein as well as techniques from this area.

Indeed the \lq if\rq\ direction is proven using one such technique known as ``McKean's argument'' - a martingale method that can be applied in any dimension and should therefore have the potential to be applicable for a large class of processes, see \cite{McKean}, Section 4 for an overview and further references.

The \lq only if\rq\ direction on the other hand is essentially a direct application of Theorem 5.7 (iii) in \cite{FL16}. The following proposition therefore mainly paraphrases the abovementioned result in our notation for the reader's convenience and the proof consists of assuring we consider a suitable set-up.

Define $\mathcal P = \{x,1-x,y,1-y\}$ where we abuse notation using $x$ for the map $(x,y)\mapsto x$ and similar for the other polynomials.

\begin{prop}[Theorem 5.7 (iii) in \cite{FL16}]\label{prop:thm5.7}
 Let $(X(t),Y(t))_{t \geq 0}$ be the $[0,1]^2$-valued solution to \eqref{eq:system_general}. Recall that it can be written in the form of \eqref{eq:poldiff} and that $A$ denotes the corresponding generator given in \eqref{eq:generator_general}. 
 
 For every polynomial $p \in \mathcal P$, assume there exists a vector of polynomials $h_p$ such that $a\nabla p = h_pp$. Furthermore assume the initial distribution $\mu_0$ to be such that $\mu_0((0,1))=1$.
 
Finally, let $\bar z \in [0,1]^2 \cap \left\{(x,y) \in [0,1]^2\mid p(x,y)=0\right\}$ be such that
 \begin{align*}
  Ap(\bar z) \geq 0 \qquad \text{ and } \qquad 2Ap(\bar z) - h_p(\bar z)^T\nabla p(\bar z) <0.
 \end{align*}
 Then for any $s>0$, there exists $\varepsilon >0$ such that 
 \begin{align*}
  \P^{\mu_0}\left\{\Vert (X(0),Y(0))-\bar z \Vert < \varepsilon\right\}=1 \;\; \Rightarrow \;\; \P^{\mu_0}\left\{\inf\{ t \geq 0\mid p(X(t),Y(t))=0\} \leq s\right\}>0.
 \end{align*}
\end{prop}
As mentioned, the proof of this proposition consists mainly of ensuring the set-up is as in Theorem 5.7 in \cite{FL16}.
\begin{proof}
We have already observed that the SDE \eqref{eq:system_general} is indeed a \emph{polynomial diffusion on $[0,1]^2$}. 
The set of polynomials $\mathcal P$ describes the state space of our diffusion by
\begin{align*}
[0,1]^2 = \{(x,y) \in [0,1]^2 \mid \forall p \in \mathcal P:\, p(x,y)\geq0\}
\end{align*}
as required. The only further assumption in Theorem 5.7 is the requirement that $\{t \geq 0 \mid p(X(t),Y(t)) = 0\}$ be a Lebesgue-nullset. However, as is immediate from the last paragraph of the proof, this is only required in order to allow the process to start on the boundary. Since we assume our process to start in the interior of $[0,1]^2$ a.s., this requirement is not necessary. Hence the proposition follows directly from Theorem 5.7, (iii), in \cite{FL16}. 
\end{proof}

We now turn to the proof of the theorem.

\begin{proof}[Proof of Theorem \ref{thm:boundary}] We will use the notation from Equation \eqref{eq:poldiff} and begin with a short observation helpful for both parts of the proof.

Let $p_0(x,y):=x \in \mathcal P$. For $h_{p_0}(x,y):= (\alpha^2(1-x),0)^T$ we have
 \begin{align*}
  a\nabla p_0 (x,y) & = \begin{pmatrix}
                 \alpha^2x(1-x) & 0 \\ 0 & (\alpha')^2y(1-y)
                \end{pmatrix} 
                \begin{pmatrix}
                 1 \\ 0
                \end{pmatrix} \\
                &  =  x\begin{pmatrix}
                            \alpha^2(1-x) \\ 0
                        \end{pmatrix} \\
                        & = p_0(x,y)h_{p_0}(x,y).
          \end{align*}
Similarly, let $p_1(x,y):=1-x \in \mathcal P$. For $h_{p_1}(x,y):= (-\alpha^2x,0)^T$ we have
 \begin{align*}
  a\nabla p_1 (x,y) & = \begin{pmatrix}
                 \alpha^2x(1-x) & 0 \\ 0 & (\alpha')^2y(1-y)
                \end{pmatrix} 
                \begin{pmatrix}
                 -1 \\ 0
                \end{pmatrix} \\
                &  =  (1-x)\begin{pmatrix}
                            -\alpha^2x \\ 0
                        \end{pmatrix} \\
                        &= p_1(x,y)h_{p_1}(x,y).
          \end{align*}

\textbf{Part 1 ``$\Rightarrow$'':} We begin proving the `only if' statements, as they rely on the Proposition \ref{prop:thm5.7} we just introduced. Let $\bar z:=(0,0)$. Then $p_0(\bar z)=0$,
\begin{align*}
 Ap_0(\bar z) 	& = u_2 \geq 0 \quad \text{ and }\quad 
 2Ap_0(\bar z) - h_{p_0}(\bar z)^T\nabla p_0(\bar z)  = 2u_2 - \alpha^2  < 0
\end{align*}
where the latter holds if, and only if $ 2u_2 < \alpha^2$. Hence the `only if' in $(i)$ follows by Proposition \ref{prop:thm5.7}.

In the same way consider instead $\bar z:=(1,1)$. Then $p_1(\bar z)=0$,
\begin{align*}
 Ap(\bar z) 	& = u_1 \geq 0 \quad \text{ and }\quad  
 2Ap(\bar z) - h_p(\bar z)^T\nabla p(\bar z) = 2u_1 - \alpha^2  < 0 	
\end{align*}
and again, the latter holds, if, and only if $2u_1	 < \alpha^2$. Therefore, the `only if' in $(ii)$ follows from Proposition \ref{prop:thm5.7} as well.

The analogous statements in $(iii)$ and $(iv)$ hold by symmetry. 

\textbf{Part 2 ``$\Leftarrow$'':} We now turn to the proof of the `if' statements, which is more involved and uses McKean's argument as it similar to the proof of Proposition 2.2 in \cite{LP16}. The approach is the same for all four cases, whence we can start with general observations and only check the different cases in the very end. Recall that we assumed the initial distribution $\mu_0$ to be such that $\mu_0( (0,1)^2) =1$.

Take $p \in \mathcal P$ and let $h_p$ be a vector of polynomials such that $a\nabla p = h_pp$ (we saw in Part 1 that this always exists). Choose $z_p \in \{(x,y) \in [0,1]^2\mid p(x,y)\neq0\}$ and define 
\begin{align*}
\tau_p:=\inf\{t\geq 0 \mid p(X(t),Y(t))=0\}. 
\end{align*}
Note that each of the $\tau_p$ corresponds to one of the stopping times defined before Theorem \ref{thm:boundary}, hence, we want to prove that $\P^{\mu_0}\left(\tau_p <\infty\right)=0$.

It\={o}'s formula and the identity $a\nabla p = h_pp$ yield
\begin{align*}
 \log p(X(t),Y(t))	& = \log p(X(0),Y(0)) \\
			& \qquad + \int_0^t\left(\frac{Ap(X(t),Y(t))}{p(X(t),Y(t))}-\frac{1}{2}\frac{\nabla p^Ta\nabla p(X(t),Y(t))}{p(X(t),Y(t))^2}\right)\dd s \\
			& \qquad + \int_0^t\frac{\nabla p^T \sigma(X(t),Y(t))}{p(X(t),Y(t))}\dd W_s \\
			& = \log p(X(0),Y(0)) + \underbrace{\int_0^t \frac{2Ap(X(t),Y(t))-\nabla p^Th_p(X(t),Y(t))}{2p(X(t),Y(t))}\dd s}_{=:\;P(t)} \\
			& \qquad + \underbrace{\int_0^t \frac{\nabla p^T \sigma(X(t),Y(t))}{p(X(t),Y(t))}\dd W_s}_{=:\;M(t)}
\end{align*}
for any $t < \tau_p$.  Suppose now, we find a constant $\kappa_p>0$ such that
\begin{align}\label{eq:conditionKappa}
 2Ap(x,y)-h_p^T\nabla p(x,y) \geq -2\kappa_pp(x,y) \qquad \text{ for all } (x,y) \in [0,1]^2.
\end{align}
Then $P$ is adapted and c\`adl\`ag, defined on $[0,\tau_p)$ and
\begin{align*}
 \inf_{t \in [0,\tau_p\wedge T)} P(t) \geq -\kappa_pT >-\infty
\end{align*}
Since $M$ is a continuous local martingale on $[0,\tau_p)$ with $M(0)=0$, Proposition 4.3 in \cite{McKean} implies $\tau_p=\infty$, $\P^{\mu_0}$-almost surely (for details concerning stochastic processes on stochastic intervals see for example \cite{Mai77}).

Recall $p_0(x,y)=x$ and the assumption in $(i)$ that $2u_2 -\alpha^2 \geq 0$. Set 
\begin{align*}
\kappa_0:= u_1+u_2+c-\alpha^2/2>0 
\end{align*}
(since $c>0$) and observe that then 
\begin{align*}
 2Ap_0(x,y)-h_{p_0}^T\nabla p_0(x,y)	& = x(-2u_1-2u_2-2c+\alpha^2) + y2c + 2u_2 -\alpha^2 \\
					& \geq x(-2u_1-2u_2-2c+\alpha^2) + 2u_2 -\alpha^2 \\
					& \geq -2\kappa_0x = -2\kappa p_0(x,y)\qquad \text{ for all } (x,y) \in [0,1]^2.
\end{align*}
Hence \eqref{eq:conditionKappa} holds for $p_0$ and since $\tau_{p_0} = \tau^X_0$, the proof of $(i)$ is completed.

For $(ii)$ we assumed $2u_1 -\alpha^2 \geq 0$ and will use $p_1(x,y)=1-x$. Set 
\begin{align*}
\kappa_1:= u_2+c>0, 
\end{align*}
since then
\begin{align*}
 2Ap_1(x,y)-h_{p_1}^T\nabla p_1(x,y)	& = x(2u_1+2u_2+2c-\alpha^2) - y2c - 2u_2 \\
					& \geq x(2u_1+2u_2+2c-\alpha^2) - 2c - 2u_2 \\
					& \geq -2\kappa_1(1-x) = -2\kappa_1p_1(x,y)\qquad \text{ for all } (x,y) \in [0,1]^2.
\end{align*}
Again, \eqref{eq:conditionKappa} holds for $p_1$ and the equality $\tau_{p_1}=\tau^X_1$ completes the proof of $(ii)$. 

As before, the remaining statements follow by symmetry.
\end{proof}

{\bf Acknowledgements.} The authors would like to thank S.\ Pulido for pointing out the connection to polynomial diffusions. JB and MWB supported by DFG Priority Programme 1590 ``Probabilistic Structures in Evolution'', project BL 1105/5-1, EB, AGC and MWB by Berlin Mathematical School and RTG 1845 ``Stochastic Analysis with applications in biology, finance and physics''.

\bibliographystyle{alpha}

\bibliography{Bib_WFdiffusion}

\end{document}